\documentclass[12pt]{amsart}
\usepackage{geometry,amsxtra,amssymb,tikz-cd,amsfonts} 

\newtheorem{theorem}{Theorem}[section]

\newtheorem{cor}{Corollary}[section]
\newtheorem{rem}{Remark}[section]
\newtheorem{lem}{Lemma}[section]

\newtheorem{ex}{Example}[section]
\newcommand{\bz}{\mathbb{Z}}
\newcommand{\br}{\mathbb{R}}
\newcommand{\bpp}{\mathbb{P}}
\newcommand{\bq}{\mathbb{Q}}
\newcommand{\bc}{\mathbb{C}}
\newcommand{\bg}{\mathbb{G}}

\newcommand{\pr}{\mathbf{P}}
\newcommand{\bP}{\mathbf{Pic}}

\newcommand{\dd}{\mathbf{dim}}
\newcommand{\mo}{\mathcal{O}}
\newcommand{\bd}{\mathbf{ed}}

\newcommand{\bF}{\mathbb{F}}
\newcommand{\bx}{\mathcal{X}}
\newcommand{\bm}{\mathcal{M}}
\newcommand{\bs}{\mathbf{Spec}}
\newcommand{\km}{\mathfrak{m}}
\newcommand{\cf}{\mathcal{F}}

\title{Essential dimension of moduli stack of polarized K3 surfaces}
\author{Anningzhe Gao}
\date{} 

\begin{document}

\maketitle
\tableofcontents
\section{Introduction}

We begin with some field $k$. Denote $Field_k$ the category of all fields extensions of $k$, with maps the obvious inclusions. Given a functor $F:Field_k\to Set$. For some element $\eta\in F(K)$ where $K/k$ is a field extension, we say some intermediate field $k\subseteqq L\subseteqq K$ a defining field of $\eta$ if there exists some $\theta\in F(L)$ such that the image of $\theta$ in $F(K)$ under the natural inclusion $L\to K$ is isomorphic to $\eta$. Then we can define the essential dimension of $\eta$, which we denoted by $\bd(\eta)$, to be the smallest transcendental  degree of $L$ when $L$ goes over all the defining field of $\eta$. Similarly, we define the essential dimension of $F$ to be the maximal $\bd(\eta)$ when $\eta$ goes over all elements in $F(K)$ for all $K$. We denote it by $\bd F$. We will use $\bd_k\eta$ and $\bd_k F$ is we want to specify the ground field $k$.

The natural examples of essential dimension in algebraic geometry coming from the algebraic stacks, in particular the moduli stack of some type varieties. To be precise, given an algebraic stack $\bx$ over $k$, we may consider the functor from $Field_k\to Set$ by sending a field extension $K$ of the set of isomorphism classes of $\bx(K)$. Then we define the essential dimension of the algebraic stack to be the essential dimension of the functor, we denote it by $\bd\bx$. We will give the basic facts of essential dimension of stacks in section 2. 

In this paper we will consider the essential dimension of the moduli stack of polarized K3 surfaces with degree $d$, we will use $\bm_d$ to denote this stack (be careful here since the degree is always an even number so some references will use $\bm_{2d}$, which is a little different here). The main theorems here are the following:

\begin{theorem}
Over $\bc$, the essential dimension of $\bm_d$ is 20 if $d=2$ and 19 if $d\geq4$.
\end{theorem}

\begin{theorem}
Let $\bF$ be an algebraic closure of some finite field $\bF_p$ for some $p>23$. Then the essential dimension of $\bm_d$  is 19 if $d\geq4$.
\end{theorem}

The degree 2 case is proved by Angelo Vistoli through an email with the author, we will give the proof in Section 3. In Section 4 we will give the proof in higher degree case. In the last section we will consider positive characteristic case.

We will call a Deligne-Mumford stack just a DM stack.\\
$\mathbf{Acknowledgment} $: I would like to thank my advisor Martin Olsson who introduced me this topic. Prof. Angelo Vistoli gives the author lots of advice for the proof, in particular his idea in the case of degree 2. Prof. Max Lieblich suggests the author to consider deformation theory in positive characteristic case. Prof. Martin Olsson helps me with the technical part of the proof. The author is really grateful with their kind help. 

\section{Basic facts about the essential dimension of algebraic stacks}

In this section we will give the basic theory of essential dimensions of algebraic stacks, the main reference here is \cite{Vistoli2007essential}. For the basic theory of stacks we refer to \cite{olssonstack}.

As we introduced in the previous section, we may define the essential dimension of a given algebraic stack. In particular if we have a scheme $X$, for simplicity we assume it is integral. 

\begin{ex}
For an integral scheme $X$ over $k$, the functor defined above is just sending a field extension $K$ to the $K$ points of $X$, i.e. $K\to X(K)$. But we know any morphism from the spectrum of a field to a scheme is always factored through $k(x)$, where $x$ is the image of $\bs K$, so it is easy to conclude that $\bd X=\dd X$ in this case.
\end{ex}

The main tool here we need to use is the following theorem, which relates the essential dimension of an algebraic stack with its generic gerbe. Recall that if a Deligne-Mumford stack has finite inertia stack, then by Keel-Mori theorem it has a coarse moduli space.

\begin{theorem}\label{gen}
Let $\bx$ be a smooth tame connected DM stack locally of finite type with finite inertia stack. Then it has a coarse moduli space $X$ with a morphism $\pi:\bx\to X$. Denote the function field of $X$ to be $K$, and $\bx_K=\bx\times_X specK$. Then we have
$$\bd_k\bx=\dd X+\bd_K\bx_K$$
\end{theorem}

For proof, see \cite{Vistoli2007essential} Theorem 4.1 in the case of char 0, and see \cite{brosnan2009essential} Theorem 6.1 for the general case. The following two corollaries are direct from Theorem \ref{gen}.

\begin{cor}
Given $\bx$ as in the theorem, for any open substack $\mathcal{U}$ of $\bx$, we have
$$\bd\bx=\bd\mathcal{U}$$
\end{cor}

\begin{cor}
Given $\bx$ as above. If the generic point only has trivial automorphism, then $$\bd\bx=\dd X$$
\end{cor}

The fiber product $\bx_K$ is a gerbe in most of cases. So we just need to consider the essential dimension of a gerbe over $\bs K$ for  a field $K$. Of course this is hard in general, but in some special cases we do have an answer. We need to use the Brauer groups. For details, see \cite{Brsurvey}

Let $G$ be a sheaf of abelian groups in the category of schemes over $K$. We say a gerbe $\bx$ over $K$ is banded by $G$ if for any scheme $T$, and object $\eta\in\bx(T)$ there is an isomorphism $G_T\cong Aut_T(\eta)$ with is compatible with the obvious pull-backs. And we know that gerbes banded by $G$ over $K$ can be classified by the cohomology group $H^2(K,G)$. In the special case when $G=\bg_m$ we get the Brauer group $Br(K)=H^2(K,\bg_m)$. We have the following injection
$$H^2(K,\mu_n)\hookrightarrow H^2(K,\bg_m)$$
and the boundary map
$$\partial:H^1(K,PGL_n)\to H^2(K,\mu_n)$$
induced by the exact sequence
$$1\to \mu_n\to SL_n\to PGL_n\to1$$
For an element $\alpha\in H^2(K,\bg_m)$, we define its index $ind(\alpha)$ tob e the smallest $n$ such that $\alpha$ is in the image of $H^1(K,PGL_n)$ under the boundary map $\partial$. Then we have the following theorem, which is highly non-trivial
\begin{theorem}(\cite{Vistoli2007essential}, Theorem 5.4)\label{tool}
Let $\bx$ be a gerbe over $K$ banded by $\mu_{p^m}$ where $p$ is a prime number, then we have
$$\bd\bx=ind([\bx])$$
here $[\bx]$ represents the class of $\bx$ in $H^2(K,\mu_{p^m})$
\end{theorem}

\section{The case when $d=2$}

Remember $k=\bc$

For a K3 surface $X$ over $k$, we mean $X$ a smooth proper variety over $k$ with trivial canonical line bundle and $H^1(X,\mo_X)=0$. For a line bundle $L$ on $X$, we say $L$ is primitive if there exists no line bundle $M$ on $X$ such that $L=M^{\otimes d}$ for some $d>1$. A polarized K3 surface od degree $d$ is a pair $(X,L)$ with $X$ a K3 surface and $L$ a primitive ample line bundle with $L^2=d$. Note here $d$ must be an even number.

We will omit the details for the moduli stack (space) of polarized K3 surfaces with degree $d$, for details, see \cite{huybrechts2016lectures} and \cite{rizov2005moduli}.

We use $\bm_d$ to denote the moduli stack of polarized K3 surfaces with degree $d$. We have 

\begin{theorem}\label{K3moduli}
The moduli stack $\bm_d$ is a connected smooth DM stack with finite inertia stack locally of finite type.
\end{theorem}

So we have a coarse moduli space $M_d$. It is integral of dimension 19, but it is not smooth.

Next we consider the case when $d=2$. We first list some properties of polarized K3 surfaces $(X,L)$ with degree 2 as an example here.

\begin{ex}\label{2}
Given a polarized K3 surface $(X,L)$ with $(L^2)=2$. Then we know that the global sections of $L$ gives a double cover $\pi: X\to \pr^2$, and $\pi$ is ramified along a smooth curve $C\subseteqq\pr^2$ os degree 6. And we know that a general curve of degree 6 in $\pr^2$ has trivial automorphism, then by the standard description of double covers we know that the automorphism of a general polarized K3 surface $(X,L)$ of degree 2 has automorphism group $\bz/2\bz$. So we know that in this case if we denote the field of rational functions of $M_2$ by $K$, then the fiber product $\bm_{2,K}=\bm_2\times_{M_2}\bs K$ is a $\bz/2\bz$ gerbe over $K$.
\end{ex}

So with the previous example and Theorem \ref{gen} we see to find the essential dimension of $\bm_2$ over $k$ we just need to consider the essential dimension of $\bm_{2,K}$ over $K$, then by Theorem \ref{tool} we just need to find the index of the class of $[\bm_{2,K}]$ in $H^2(K,\mu_2)$.

\begin{theorem}
The gerbe $\bm_{2,K}$ is a trivial gerbe, so it has index 1, hence $\bd \bm_2=20$.
\end{theorem}
\begin{proof}
From Example \ref{2} we know that the moduli space $M_2$ is just the quotient of the space of smooth curves of degree 6 in $\pr^2$ by the $PGL_3$ action. So there exists a Brauer-Severi surface $P$ over $K$ and a degree 6 curve $C\subseteqq P$. Then the generic gerbe is just the stack of the double covers $Y\to P$ which ramified over $C$. To prove the gerbe is trivial it suffices to show $\bm_{2,K}(K)$ is non-empty. But by the description of double covers we just need to find a square root of $\mo(C)$. Then since $\bP(P)\cong\bz$, we just need to find a line bundle of degree 3. But the inverse of the canonical line bundle has degree 3, so the gerbe is trivial, we finish the proof. 
\end{proof}

This gives the answer in the case of $d=2$.

\section{The case when $d>2$}

Next we consider the case when the degree is greater than 2. We first proof a technical lemma.

\begin{lem}\label{tech}
Let $\bx$ be a smooth connected DM stack with finite inertia locally of finite type over $k$. If there exists a $k$ point $x$ on $\bx$ such that the automorphism group $G_x$ is trivial, then there exists an open dense subset $\mathcal{U}$ of $\bx$ such that for any point on $\mathcal{U}$, it will have trivial automorphism group, hence the generic point has trivial automorphism group.
\end{lem}
\begin{proof}
The automorphism group $G_x$ is defined by the pull-back
$$G_x=\bs k\times_{\bx}\mathcal{I}_{\bx}$$ 
where $\mathcal{I}_{\bx}$ is the inertia stack of $\bx$. We pick some etale cover $U\to\bx$ of $\bx$, then we have $U$ is smooth. Then there exists a $k$ point $u$ on $U$ such that the composition 
$$u\to U\to\bx$$ is just the given $k$ point (here we use that $k$ is algebraically closed). We may choose an affine connected open neighborhood of $u$, also denoted by $U$. Write 
$$G_U=U\times_{\bx}\mathcal{I}_{\bx}$$
Since $\mathcal{I}_{\bx}$ is finite over $\bx$, by definition $G_U$ is also an affine scheme over $k$. Also we know for any dense open subset of $U$ its image in $\bx$ will be a dense open set. So we just need to prove that there exist an open dense sub scheme $W$ of $U$ such that for any point of $W$ the fiber is just $\bs k$.

Let's denote the morphism $G_U\to U$ by $\pi$, and the point with trivial automorphism group by $x$. Write $G_U=\bs B$, $U=\bs A$, the maximal ideal corresponding to $x$ by $\km$, then we have $B/\km B\cong k$. Consider the sheaf associated to $B$ on $U$, we call it $\cf$. Then $\cf$ is coherent. We have 
$$\dd \cf_x\otimes\mo_{U,x}/\km_x=1$$
But by \cite{AG} Ex. 5.8 on page 125, the function
$$\phi(u)=\dd\cf_u\otimes\mo_{U,u}/\km_u$$
is an upper semi-continuous function. Since the fiber over any point is non-empty, so the set
$$W=\{u\in U|\phi(u)=1\}$$
is an open sub scheme of $U$. Since $U$ is integral, $W$ is dense.

By the same exercise, since $U$ is integral, so on $W$, $\cf_W$ is a line bundle. Then we may choose an open subscheme of $W$ (just shrink $W$, so we still use $W$ to denote it) such that on it we have $\cf_W\cong\mo_W$. Then on this subset, it is obvious that every point has fiber $\bs k$. This dense open sub scheme satisfies the property we want.
\end{proof}

By the above lemma, since $\bm_d$ satisfies all the property we need, if we find a polarized K3 surface with trivial automorphism, then the generic point must has trivial automorphism group. But we know there must exist a polarized K3 surface with Picard number 1 of any degree, then by the following theorem, we can get our result.

\begin{theorem}\label{111}
If $X$ is a complex projective K3 surface with $\bP(X)=\bz H$ and $(H^2)>2$, then $Aut(X)=id$.
\end{theorem}

For proof, see \cite{huybrechts2016lectures}, Corollary 12.2.12.

Combine this theorem and the above technical lemma and Theorem \ref{gen}, we conclude

\begin{theorem}
The essential dimension of $\bm_d$ is 19 when $d>2$.
\end{theorem}

\section{The case of positive characteristic}
 We will consider the moduli stack of polarized K3 surfaces in (large) positive characteristic in this section. Actually we will only consider the K3 surfaces over $\bF$, the algebraic closure of $\bF_p$ for some prime $p\geq23$ and $p\nmid d$. The idea is to apply the deformation theory of K3 surfaces and the result we got above. We first collect some results we need.
 
Recall that actually to show the essential dimension of some DM stack with finite inertia actually we just need to consider the "general case", so it suffices to consider the ordinary K3 surfaces. It suffices to find one point in the stack with trivial automorphism group, so let's concentrate on the K3 surfaces with height 1 and Picard number 20. Over $\bc$, K3 surfaces with Picard number 20 can be classified by the transcendental lattices, which have rank 2. More precisely, given any positive definite, oriented even lattice $M$ of rank 2, there exists a unique complex K3 surface $X_M$ with its transcendental lattice is just $M$. Let $S_p$ denote the set of positive definite even lattices with rank 2 such that the discriminant if a non-zero square mod $p$. Then for any such lattice $M$, there exists a complex K3 surface $X$ over $\bc$ with Picard number 20 and transcendental lattice $M$. Such $X$ is defined over $\bar{\bq}$. The reduction of $X$ over $\bF$ is a K3 surface with Picard number 20 (For details see the discussion in \cite{jang2013neron} section 3). And we have the following:

\begin{theorem}\label{classify}(\cite{jang2013neron}, Theorem 3.7)
The K3 surfaces over $\bF$ with Picard number are classified by $S_p$.
\end{theorem}

For an ordinary K3 surface over $\bF$ with Picard number 20, it has a unique Neron-Severi preserved lifting, which is just the canonical lifting \cite{nygaard1983tate}. We have the following theorem comparing the automorphism group of the K3 surface itself and the Neron-Severi preserved lifting.

\begin{theorem}\label{aut}(\cite{jang2016lifting}, Theorem 3.7)
Let $X$ be a weakly tame K3 surface over $\bF$, then there exists a Neron-Severi group preserving lifting $\mathfrak{X}/W$, where $W$ is the Witt vector of $\bF$, such that the reduction map $Aut(\mathfrak{X}\otimes K)\to Aut(X)$ is an isomorphism. 
\end{theorem}

For the proof and the definition of "weakly tame", we refer to \cite{jang2016lifting}. Recall that if $p>22$, then every K3 surface of finite height is weakly tame. And by the standard argument we can show that in this case $Aut(\mathfrak{X}\otimes\bc)\cong Aut(X)$ and $NS(\mathfrak{X}\otimes\bc)\cong NS(X)$ canonically.

Now let's return to the polarized case.  Given the moduli stack $\bm_d$ of polarized K3 surfaces over $\bF$, let's assume the degree of the polarization is greater than 2. By Lemma \ref{tech}, if we can find a point on $\bm_d$ with trivial automorphism group, we can show the generic point has trivial automorphism group. 

We next need to use the theory of period domain. We suggest \cite{huybrechts2016lectures}, chapter 6 for the basic facts and properties. We have the following theorem:

\begin{theorem}
We fix the ground field to be $\bc$. In the moduli stack $\bm_d$, the set of K3 surfaces with Picard number 20 and the discriminant of the transcendental lattice is a non-zero square mod $p$ form a dense subset of $\bm_d(\bc)$.
\end{theorem}
Before we start our proof, let's first recall some basic result about the period domain. Set the lattice $\Lambda_d=\bz(-d)\oplus U^{\oplus2}\oplus E_8(-1)^{\oplus2}$ and $(,)$ the quadratic form on it, where $U$ is the hyperbolic lattice and $E_8$ is the lattice associated to the Dynkin diagram $E_8$. This is the lattice orthogonal to the lattice generated by the ample line bundle of degree $d$ we choose in the standard K3 lattice, see \cite{huybrechts2016lectures} chapter 6.

So we have the period domain of marked polarized K3 surfaces with degree $d$, which we denoted by $D_d$. $D_d$ is a subset of $\bpp(\Lambda_{d\bc})$. The group of orthogonal matrix $O_d$ has a natural action on $D_d$, and we know that the coarse moduli space $M_d$ is just an open subset of the quasi-projective variety $D_d/O_d$. So to prove the above theorem, we just need show the points on $D_d$ corresponding to the mark polarized K3 surfaces with transcendental lattice has non-zero square mod p discriminant is dense in $D_d$. But since $D_d$ is diffeomorphic to the set of oriented planes in $\Lambda_d\otimes\br$ such that the restriction of the quadratic form on the plane is positive definite (\cite{huybrechts2016lectures}, Proposition 6.1.5), and the set of positive definite planes is an open subset of the set of all planes, we just need to prove the following theorem, for simplicity, we call the property of a lattice "with discriminant a non-zero square mod $p$" just property $R$.

\begin{theorem}\label{22}
The set of rationally generated planes in $\Lambda_d\otimes\br$ satisfying property $R$ forms a dense subset of the grassmannian $Gr(2,\Lambda_d\otimes\br)$.
\end{theorem}
\begin{proof}
Choose an open subset of $Gr(2,\Lambda_d\otimes\br)$. We first choose a rationally generated plane $H\in U$. And we notice that actually to check property $R$, it suffices to check the discriminant of a $rational\ basis$ of $H\cap \Lambda_d$. The reason is that the discriminant of such basis only differ with the discriminant of the integral basis by a square, so if the discriminant of a rational basis is a non-zero square mod $p$, the same is true for the integral basis. Also, when we write $\frac{1}{N}\in\bz/p$ for some $p\nmid N$, we just mean the inverse of $N$ in $\bz/p$.

Assume $H\cap\Lambda_d$ is rationally generated by $\omega_1,\omega_2$ in $\Lambda_d$. We first need to do some reductions:

$\mathbf{Step 1}$. We may assume $(\omega_1,\omega_1)\neq0$ in $\bz/p$. If not, we can consider the element $\delta\in\Lambda_d$ satisfying $(\delta,\delta)=-2$. Of course $\delta$ in not in $H$. Then consider the plane $H'$ generated by $(\omega_1+\frac{1}{N}\delta,\omega_2)$. For $N$ large enough, $H'$ is in $U$. And $$(\omega_1+\frac{1}{N}\delta,\omega_1+\frac{1}{N}\delta)=(\omega_1,\omega_1)+\frac{2}{N}(\omega_1,\delta)+\frac{1}{N^2}(\delta,\delta)$$
since $(\delta,\delta)\neq0\in\bz/p$, there must exists some $N$ such that $(\omega_1,\omega_1)+\frac{2}{N}(\omega_1,\delta)+\frac{1}{N^2}(\delta,\delta)\neq0$ in $\bz/p$. We we can replace $H$ by $H'$ to assume that $(\omega_1,\omega_1)\neq0$ in $\bz/p$.

$\mathbf{Step 2}$. We may assume $disc(\omega_1,\omega_2)$ is non-zero mod $p$. Consider $H'$ generated by $(\omega_1,\omega_2+\frac{1}{N}\delta)$ with any $(\delta,\delta)\neq0$ in $\bz/p$. Then we have $disc(\omega_1,\omega_2+\frac{1}{N}\delta)=((\omega_1,\omega_1)(\omega_2,\omega_2)-(\omega_1,\omega_2)^2)+\frac{2}{N}((\omega_1,\omega_1)(\omega_2,\delta)-(\omega_1,\delta)(\omega_1,\omega_2))+\frac{1}{N^2}((\omega_1,\omega_1)(\delta,\delta)-(\omega_1,\delta)^2)$. By the same reason, we just need to find some $\delta$ such that the leading coefficient $(\omega_1,\omega_1)(\delta,\delta)-(\omega_1,\delta)^2$ is non-zero in $\bz/p$. We prove the existence of such $\delta$ by contradiction. If for any $(\delta,\delta)\neq0$ in $\bz/p$, we have $(\omega_1,\omega_1)(\delta,\delta)-(\omega_1,\delta)^2$ is zero in $\bz/p$. Since $\Lambda_d$ contains $E_8(-1)$, so we can find $\delta_1,\delta_2$, with $(\delta_1,\delta_2)=0$ and $(\delta_i,\delta_i)=-2$ for $i=1,2$. We have $$(\omega_1,\omega_1)(\delta_i,\delta_i)-(\omega_1,\delta_i)^2$$ is zero for both $i$, and $(\omega_1,\omega_1)(\delta_1+\delta_2,\delta_1+\delta_2)-(\omega_1,\delta_1+\delta_2)^2$ is zero, we may conclude
$$(\omega_1,\omega_1)(\delta_1,\delta_2)=(\omega_1,\delta_1)(\omega_1,\delta_2)$$ in $\bz/p$. But $(\delta_1,\delta_2)=0$, so we may assume $(\omega_1,\delta_1)=0$ in $\bz/p$. But then $$(\omega_1,\omega_1)(\delta_1,\delta_1)-(\omega_1,\delta_1)^2\equiv(\omega_1,\omega_1)(\delta_1,\delta_1)$$
which is non-zero, which is a contradiction. By replacing $H$ by $H'$, we may assume the discriminant is non-zero.

$\mathbf{Step 3}$. We may assume there exists a $\eta\in\Lambda_d$ orthogonal to $H$ and $(\eta,\eta)$ is non-zero in $\bz/p$. From the previous two steps, we may assume $H$ is rationally generated by $\omega_1,\omega_2$ with $(\omega_i,\omega_i)\neq0$ in $\bz/p$ for $i=1,2$ and $(\omega_1,\omega_2)=0$ (just by diagonalizing the matrix). Pick any $\delta\in\Lambda_d$, we have $$\delta-\frac{(\omega_1,\delta)}{(\omega_1,\omega_1)}\omega_1-\frac{(\omega_2,\delta)}{(\omega_2,\omega_2)}\omega_2$$ is orthogonal to $H$. And
$$(\delta-\frac{(\omega_1,\delta)}{(\omega_1,\omega_1)}\omega_1-\frac{(\omega_2,\delta)}{(\omega_2,\omega_2)}\omega_2,\delta-\frac{(\omega_1,\delta)}{(\omega_1,\omega_1)}\omega_1-\frac{(\omega_2,\delta)}{(\omega_2,\omega_2)}\omega_2)$$
$$=(\delta,\delta)-\frac{(\omega_1,\delta)^2}{(\omega_1,\omega_1)}-\frac{(\omega_2,\delta)^2}{(\omega_2,\omega_2)}$$ 

If for some $\delta$ the above number is non-zero in $\bz/p$, we are done. If not, let's choose a $\eta$ orthogonal to $H$, also let's assume $\eta$ is primitive. Then $H'$ defined by $(\omega_1,\omega_2+\frac{1}{N}\eta)$ also satisfies the assumption we made in step 1 and step 2. If for this plane, we also have every $\delta$ orthogonal to $H'$ has $(\delta,\delta)\equiv0$, then we have
$$(\delta,\delta)-\frac{(\omega_1,\delta)^2}{(\omega_1,\omega_1)}-\frac{(\omega_2+\frac{1}{N}\eta,\delta)^2}{(\omega_2,\omega_2)}$$ is zero in $\bz/p$. Comparing with the previous equation we get $$(\eta,\delta)\equiv0$$ for any $\delta$. But this makes $\eta$ is divided by $p$ in $\Lambda_d$, which is a contradiction to the primitivity of $\eta$. So by replacing $H$ by $H'$ we may assume there is a $\eta$ orthogonal to $H$ with $(\eta,\eta)\neq0$ in $\bz/p$.

$\mathbf{Step 4}$. We finish the proof in this step. So far we have a plane $H\in U$ rationally generated by $\omega_1,\omega_2$ with $(\omega_i,\omega_i)\neq 0$ in $\bz/p$ and $(\omega_1,\omega_2)=0$. Set $A=disc(\omega_1,\omega_2)$. Then $A\neq0$ in $\bz/p$ in step 2. If $A$ is already a square, we are done. If not, by step 3, we choose a $\eta$ orthogonal to $H$ and $(\eta,\eta)\neq0$ in $\bz/p$. Consider $H'$ generated by $\omega_1,\omega_2+\frac{1}{N}\eta$ for large enough $N$. Then $$disc(\omega_1,\omega_2+\frac{1}{N}\eta)=A+\frac{1}{N^2}(\omega_1,\omega_1)(\eta,\eta)$$
Denote $B=-(\omega_1,\omega_1)(\eta,\eta)$, then $B$ is non-zero in $\bz/p$. And $$disc(\omega_1,\omega_2+\frac{1}{N}\eta)=A-By^2$$
here $y$ is the inverse of $N$ mod $p$ and can be 0, which means $H=H'$. We just need to show for some $y$, $A-By^2$ is a square mod $p$. We separate into 2 cases:\\
(1) If $B$ is a square mod $p$, consider the set $S=\{0,1,2,...,\frac{p-1}{2}\}$. Then for any $y_1,y_2\in S$, $A-By_1^2\neq A-By_2^2$ mod $p$ unless $y_1=y_2$. But $A-By^2$ cannot be 0 in $\bz/p$ since $A$ is not a square. But we only have $\frac{p-1}{2}$ non-squares mod $p$, and $S$ contains $\frac{p+1}{2}$ elements. So there must be some $y$ that makes $A-By^2$ a non-zero square.\\
(2) If $B$ is not a square, then $A/B$ is a square. We prove by contradiction. If for any $y$, $A-By^2$ is zero or not a square mod $p$, then for any $y$, $\frac{A}{B}-y^2$ is a square (maybe 0), hence $1-y^2$ is a square for any $y$. The following is a little tricky: We notice $-1$ cannot be a square, or $y^2-1$ is a square then by induction every element in $\bz/p$ is a square, which is not true. Then $2$ cannot be a square, otherwise $1-2=-1$ is a square. Then $-2=-1\times2$ is a square. So $1-(-2)=3$ is a square. On the other hand, $\frac{1}{y^2}-1$ is a square for any $y\neq0$, in particular $-2-1=-3$ is a square, this makes $-1=\frac{-3}{3}$ is a square, which is a contradiction. We finish the proof.
\end{proof}

So by the above theorem, and Theorem \ref{111}, we can find a polarized K3 surface over $\bc$ with trivial automorphism group such that the transcendental lattice $T$ is in $S_p$. It's reduction is a polarized K3 surface $(X,L)$ over $\bF$. Then from Theorem \ref{classify} and \ref{aut}, we can conclude that $Aut(X,L)=\{id\}$. Then by Theorem \ref{tech}, we can conclude that the generic object of $\bm_d$ in char $p$ case also has trivial automorphism group if $d>2$.

In \cite{dolgachev2009finite}, it is proved that if $p>11$, then every automorphism of a K3 surface with finite order is tame, in particular, that means in that case the moduli stack $\bm_d$ is tame, so apply Theorem \ref{gen}, we have

\begin{theorem}
Let's assume $\bF$ is the algebraic closure of $\bF_p$ for $p\geq23$. Then the essential dimension of $\bm_d$ with $d>2$ and $p\nmid d$ is 19.
\end{theorem}

\begin{rem}
If we only want to show the existence of a polarized K3 surface with trivial automorphism group, we can consider the moduli stack $\bm_d$ over $\bs(\bz[\frac{1}{2d}])$ and compare the special fiber and generic fiber. But here we have a stronger statement.
\end{rem}

\bibliographystyle{abbrv}
\bibliography{MyCitation}

\begin{thebibliography}{10}

\bibitem{Vistoli2007essential}
P.~Brosnan, Z.~Reichstein, and A.~Vistoli.
\newblock Essential dimension and algebraic stacks.
\newblock {\em arXiv preprint math/0701903}, 2007.

\bibitem{brosnan2009essential}
P.~Brosnan, Z.~Reichstein, A.~Vistoli, and N.~Fakhruddin.
\newblock Essential dimension of moduli of curves and other algebraic stacks.
\newblock {\em arXiv preprint arXiv:0907.0924}, 2009.

\bibitem{dolgachev2009finite}
I.~V. Dolgachev and J.~Keum.
\newblock Finite groups of symplectic automorphisms of k3 surfaces in positive
  characteristic.
\newblock {\em Annals of mathematics}, pages 269--313, 2009.

\bibitem{AG}
R.~Hartshorne.
\newblock {\em Algebraic geometry}, volume~52.
\newblock Springer Science \&amp; Business Media, 2013.

\bibitem{huybrechts2016lectures}
D.~Huybrechts.
\newblock {\em Lectures on K3 surfaces}, volume 158.
\newblock Cambridge University Press, 2016.

\bibitem{Brsurvey}
J.~Jahnel.
\newblock {\em The Brauer-Severi variety associated with a central simple
  algebra: A survey}.
\newblock Citeseer, 2003.

\bibitem{jang2013neron}
J.~Jang.
\newblock Neron-severi group preserving lifting of k3 surfaces and
  applications.
\newblock {\em arXiv preprint arXiv:1306.1596}, 2013.

\bibitem{jang2016lifting}
J.~Jang.
\newblock A lifting of an automorphism of a k3 surface over odd characteristic.
\newblock {\em International Mathematics Research Notices}, 2017(6):1787--1804,
  2016.

\bibitem{nygaard1983tate}
N.~O. Nygaard.
\newblock The tate conjecture for ordinary k3 surfaces over finite fields.
\newblock {\em Inventiones mathematicae}, 74(2):213--237, 1983.

\bibitem{olssonstack}
M.~Olsson.
\newblock {\em Algebraic spaces and stacks}, volume~62.
\newblock American Mathematical Soc., 2016.

\bibitem{rizov2005moduli}
J.~Rizov.
\newblock Moduli stacks of polarized k3 surfaces in mixed characteristic.
\newblock {\em arXiv preprint math/0506120}, 2005.

\end{thebibliography}

\end{document}